\documentclass[american, 5p]{elsarticle}
\pdfoutput=1
\biboptions{longnamesfirst, sort&compress}

\usepackage[utf8]{inputenc}
\usepackage[T1]{fontenc}
\usepackage{lmodern}
\usepackage[babel=true]{microtype}

\usepackage{babel}
\usepackage{amsfonts, amsmath, amsopn, amssymb, amsthm, bm, braket}
\usepackage{enumitem}

\usepackage{pgfplots}
\pgfplotsset{compat=newest}
\usetikzlibrary{calc}

\PassOptionsToPackage{hyphens}{url}
\usepackage{hyperref}
\hypersetup{%
   pdfauthor={Weizhang Huang, Lennard Kamenski, Jens Lang},
   pdfsubject={2010 MSC: 65M60, 65M50, 65F15},
   pdftitle={Stability of explicit Runge-Kutta methods
      for high order finite element approximation
      of linear parabolic equations},
}
\usepackage[capitalize]{cleveref}
\crefname{equation}{}{}
\crefname{enumi}{}{}

\newtheorem{theorem}{Theorem}
\newtheorem{lemma}{Lemma}
\newtheorem{corollary}{Corollary}
\theoremstyle{definition}
\newtheorem{remark}{Remark}
\newtheorem{example}{Example}

\DeclareMathOperator{\diag}{diag}
\DeclareMathOperator{\meas}{meas}
\providecommand{\from}{\colon}
\newcommand{\Diff}{\mathbb{D}}
\newcommand{\dx}{\;\mathrm{d}\bm{x}}
\newcommand{\dxi}{\;\mathrm{d}\bm{\xi}}
\newcommand{\be}{\bm{e}}
\newcommand{\bu}{\bm{u}}
\newcommand{\bU}{\bm{U}}
\newcommand{\bv}{\bm{v}}
\newcommand{\bx}{\bm{x}}
\newcommand{\bxi}{\bm{\xi}}
\newcommand{\MA}{M^{-1} A}
\newcommand{\tMA}{\tM^{-1} A}
\newcommand{\FF}{{(F_K')}^{-1} {(F_K')}^{-T}}
\newcommand{\FDF}{{(F_K')}^{-1}\Diff{(F_K')}^{-T}}
\newcommand{\Abs}[1]{{\left\lvert#1\right\rvert}}
\newcommand{\abs}[1]{{\lvert#1\rvert}}
\newcommand{\Norm}[1]{{\left\lVert#1\right\rVert}}
\newcommand{\norm}[1]{{\lVert#1\rVert}}
\newcommand{\NormE}[1]{{|||#1|||}}
\newcommand{\tM}{\tilde{M}}
\newcommand\Th{\mathcal{T}_h}

\begin{document}

\date{}

\begin{frontmatter}

\journal{Numerical Mathematics and Advanced Applications --- ENUMATH 2013}
\title{%
   Stability of~explicit Runge-Kutta methods
   for~high order finite element approximation
   of~linear parabolic equations\tnoteref{t1}%
}
\tnotetext[t1]{%
 The work was supported in part by the NSF (U.S.A.) under Grant DMS-1115118,
the Excellence Initiative of the German Federal and State Governments,
and the Graduate School of Engineering at the Technische Universit{\"a}t Darmstadt.%
}

\author[addressKU]{Weizhang Huang}
\ead{whuang@ku.edu}

\author[addressWIAS]{Lennard Kamenski}
\ead{kamenski@wias-berlin.de}

\author[addressCSI]{Jens Lang}
\ead{lang@mathematik.tu-darmstadt.de}

\address[addressKU]{%
   Department of~Mathematics, University of~Kansas, Lawrence, KS~66045, USA}
\address[addressWIAS]{%
   Weierstrass Institute for Applied Analysis and Stochastics, Berlin, Germany}
\address[addressCSI]{%
   Department of~Mathematics, Graduate School of~Computational Engineering,
   and~Center of~Smart Interfaces, TU Darmstadt, Germany}


\begin{abstract}
We study the stability of explicit Runge-Kutta methods for high order Lagrangian finite element approximation of linear parabolic equations and establish bounds on the largest eigenvalue of the system matrix which determines the largest permissible time step.
A bound expressed in terms of the ratio of the diagonal entries of the stiffness and mass matrices is shown to be tight within a small factor which depends only on the dimension and the choice of the reference element and basis functions but is independent of the mesh or the coefficients of the initial-boundary value problem under consideration.
Another bound, which is less tight and expressed in terms of mesh geometry, depends only on the number of mesh elements and the alignment of the mesh with the diffusion matrix.
The results provide an insight into how the interplay between the mesh geometry and the diffusion matrix affects the stability of explicit integration schemes when applied to a high order finite element approximation of linear parabolic equations on general nonuniform meshes.
\end{abstract}

\begin{keyword}%
   finite element method, anisotropic mesh, stability condition, parabolic equation
   \MSC[2010]
   65M60, 
   65M50, 
   65F15  
\\[0.6\baselineskip]%
{\footnotesize{%
This is a preprint of a contribution to
A.~Abdulle et al.~(eds.),
\emph{Numerical Mathematics and Advanced Applications --- ENUMATH~2013},
Lecture Notes in Computational Science and Engineering, vol.\ 103.
\\%
\textcopyright~Springer International Publishing Switzerland 2015.
The final version is available at \url{https://doi.org/10.1007/978-3-319-10705-9_16}.
}}%
\end{keyword}

\end{frontmatter}

\section{Introduction}
We consider the initial-boundary value problem (IBVP)
\begin{align}
   \begin{cases}
   \begin{alignedat}{3}
      & u_t = \nabla \cdot \left(\Diff \nabla u \right),
      && \qquad \bx \in \Omega,
      &\quad t \in \left( 0, T \right],
      \\
      & u(\bx, t) = 0,
      && \qquad \bx \in \Gamma_D,
      &\quad t \in \left( 0, T \right],
      \\
      & \Diff \nabla u(\bx, t) \cdot \bm{n} = 0,
      && \qquad \bx \in \Gamma_N,
      &\quad t \in \left( 0, T \right],
      \\
      & u(\bx, 0) = u^0(\bx),
      && \qquad \bx \in \Omega,
   \end{alignedat}
   \end{cases}
   \label{eq:IBVP}
\end{align}
where $\Omega \subset \mathbb{R}^d$ ($d \ge 1$) is a bounded polygonal or polyhedral domain, $\Gamma_D \cup \Gamma_N = \partial \Omega$, $\meas_{d-1} \Gamma_D > 0$, $u^0$ is a given function, and $\Diff = \Diff(\bx)$ is the diffusion matrix, which is assumed to be time-independent, symmetric and uniformly positive-definite on $\Omega$.
If $u^0\in H^1_{D}(\Omega) = \set{v \in H^1(\Omega) : \text{$v = 0$ on $\Gamma_D$}}$ and $u$ is sufficiently smooth, then the solution of the IBVP satisfies the stability estimates
\begin{align*}
   \begin{cases}
   \begin{alignedat}{2}
      \Norm{u(\cdot,t)}_{L^2(\Omega)} 
      &\le \Norm{u^0}_{L^2(\Omega)},
         && \qquad t \in \left( 0, T \right],\\[2mm]
         \NormE{u(\cdot,t)}
      &\le \NormE{u^0},
         && \qquad t \in \left( 0, T \right],
   \end{alignedat}
   \end{cases}
\end{align*}
where $\NormE{u} =\Norm{\Diff^{1/2}\nabla u}_{L^2(\Omega)}$ is the energy norm.
We are interested in the stability conditions so that the numerical approximation preserves these stability estimates.

The stability of explicit Runge-Kutta methods depends on the largest eigenvalue of the corresponding system matrix, which, in turn, depends on the mesh and the coefficients of the IBVP.\@
For our model problem this means that we need to estimate the largest eigenvalue of $\MA$, where $M$ and $A$ are the mass and stiffness matrices for the finite element discretization of the IBVP \cref{eq:IBVP}~\cite[Theorem~3.1]{HuaKamLan13b}.
For the Laplace operator on a uniform mesh it is well known that $\lambda_{\max}(\MA) \sim N^{2/d}$, where $N$ is the number of mesh elements.
For general meshes and diffusion coefficients, estimates have been derived recently
in Huang et al.~\cite{HuaKamLan13b} and Zhu and Du~\cite{ZhuDu11,ZhuDu14} (see also~\cite{DuWanZhu09,Fri73,GraMcL06,KamHuaXu13} for estimates on $M$ and $A$).
All of these works allow anisotropic diffusion coefficients and anisotropic meshes, while the former employs a more accurate measure for the interplay between the mesh geometry and the diffusion matrix and gives a sharper estimate on $\lambda_{\max}(\MA)$ than the latter.
On the other hand,~\cite{HuaKamLan13b} considers only linear finite elements whereas the estimates in~\cite{ZhuDu14} are valid for both linear and higher order finite elements.

The purpose of this paper is to extend the result of~\cite{HuaKamLan13b} to high order Lagrangian finite elements
as well as provide a mathematical understanding of how the interplay between the mesh geometry and the diffusion matrix affects the stability condition.
We show that the main result of~\cite[Theorem~3.3]{HuaKamLan13b} holds for high order finite elements as well.
The analysis is based on bounds on the mass and stiffness matrices.
We follow the approach in~\cite{HuaKamLan13b,KamHuaXu13} and derive simple but accurate bounds for the case of high order Lagrangian finite elements on simplicial meshes (\cref{thm:smatrix,thm:smatrix:fdf,thm:M:W,thm:M12}).
We also consider the more general case of surrogate mass matrices $\tM$.
The main result (\cref{thm:main:result}) shows that $\lambda_{\max}(\tMA)$ is proportional to the maximum ratio between the corresponding diagonal entries of the stiffness and surrogate mass matrices.
Moreover, $\lambda_{\max}(\tMA)$ is bounded by a term depending only on the number of the mesh elements and the alignment of the shape of the mesh elements with the inverse of the diffusion matrix.

\section{Stability condition for explicit time stepping}\label{sec:preliminary}
\enlargethispage{0.5\baselineskip}

Let $\set{\Th}$ be a family of simplicial meshes for $\Omega$ and $V^h$ the Lagrangian $\mathbb{P}_m$ ($m \ge 1$) finite element space associated with $\Th$.
Let $K$ be an arbitrary element of $\Th$, $\hat{K}$ the reference element, and $\omega_i$ the element patch of the $i^{\text{th}}$ vertex (\cref{fig:notation}); element and patch volumes are denoted by $\Abs{K}$ and \mbox{$\Abs{\omega_i} = \sum_{K \in \omega_i} \Abs{K}$}.
For each $K\in \Th$ let $F_K \from \hat{K} \to K$ be an invertible affine mapping and  $F_K'$ its Jacobian matrix which is constant and satisfies $\det(F_K') = \Abs{K}$  (for simplicity, we assume that $\abs{\hat{K}} = 1$).
We further assume that the mesh is fixed for all time steps.
%
\begin{figure}[t]%
   \centering{}%
   \begin{tikzpicture}[scale = 0.75]%
      \tikzstyle{every node}=[font=\small]
      %
      %
      \path ( 0.0,  0.0) coordinate (N0); 
      \path ( 2.0,  0.0) coordinate (N1);  
      \path ( 0.6,  1.8) coordinate (N2);
      \path (-1.5,  1.3) coordinate (N3);
      \path (-1.4, -1.0) coordinate (N4);
      \path ( 0.4, -1.6) coordinate (N5);
      %
      %
      \draw [] (N0) -- (N1) -- (N2) -- cycle;
      \draw [] (N0) -- (N2) -- (N3) -- cycle;
      \draw [fill = gray!12] (N0) -- (N3) -- (N4) -- cycle;
      \draw [] (N0) -- (N4) -- (N5) -- cycle;
      \draw [] (N0) -- (N5) -- (N1) -- cycle;
      %
      %
      \path ( -5.6, -1.2) coordinate (Q1); 
      \path ( -5.6,  1.4) coordinate (Q2);  
      \path ( -3.0, -1.2) coordinate (Q3);
      \path ( -5.6,  0.1) coordinate (Q12);
      \path ( -4.3, -1.2) coordinate (Q13);
      \path ( -4.3,  0.1) coordinate (Q23);
      \draw [fill = gray!12] (Q1) -- (Q2) -- (Q3) -- cycle;
      \filldraw [black] (Q1)  circle (2.0pt);   
      \filldraw [black] (Q2)  circle (2.0pt);   
      \filldraw [black] (Q3)  circle (2.0pt);   
      \filldraw [black] (Q12)  circle (2.0pt);   
      \filldraw [black] (Q13)  circle (2.0pt);   
      \filldraw [black] (Q23)  circle (2.0pt);   
      %
      %
      %
      \path (-4.6,  -0.2) coordinate (Khat);
      \path (-0.9,  0.3) coordinate (K);
      \node [below] at (Khat)  {$\hat{K}$};
      \node [below] at (K)  {$K$};
      \node [above right]  at (N0)  {node $i$};
      \node [above] at (2.5, -1.2)  {patch $\omega_i$};
      \filldraw [black] (N0)  circle (2.0pt);
      %
      %
      \path [->, line width = 1pt] ($(Khat) + (0.3,-0.2)$)
         edge [bend left] node [above] {$F_K(\bxi)$} ($(K) + (-0.3,-0.20)$);
      \path [->, line width = 1pt] ($(K) + (-0.30,-0.50)$)
         edge [bend left] node [above] {$F_K^{-1}(\bx)$} ($(Khat) - (-0.3,0.40)$);
      %
      %
      \path ( 3.1,  0.0) coordinate (N6); 
      \path ( 4.2,  1.7) coordinate (N7);  
      \path ( 5.3,  0.0) coordinate (N8);
      \path ( 4.2, -1.6) coordinate (N9);
      \path ( 4.2,  0.0) coordinate (N10);
      \draw [] (N6) -- (N7) -- (N8);
      \draw [] (N6) -- (N8) -- (N9) -- cycle;
      \node [above]  at (N10)  {node $i$};
      \filldraw [black] (N10)  circle (2.0pt);
   \end{tikzpicture}%
   \caption{Example of the standard quadratic FE reference mesh element $\hat{K}$,
      mapping $F_K$, the corresponding mesh elements $K$,
   nodes and their patches.\label{fig:notation}}%
\end{figure}
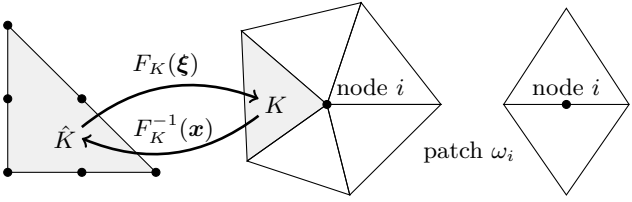

With $V^h_D=V^h\cap H^1_D(\Omega)$, the finite element solution $u^h(t) \in V^h_D$, $t \in \left( 0, T \right]$, is defined by
\begin{equation}
   \int_\Omega \partial_t u^h v^h \dx
   = - \int_\Omega \nabla v^h \cdot \Diff \nabla u^h \dx,
   \qquad \forall v^h \in V^h_D,
   \label{eq:FEM:i}
\end{equation}
subject to the initial condition
\begin{equation}
   \int_\Omega u^h(\bx, 0) v^h \dx
   = \int_\Omega u^0(\bx) v^h \dx,
   \qquad \forall v^h \in V^h_D .
   \label{eq:FEM:ii}
\end{equation}
Let $N_\phi$ be the dimension of the finite element space $V^h_D$ and denote a nodal basis of $V^h_D$ by $\set{\phi_1, \dotsc, \phi_{N_\phi}}$, then $u^h$ can be expressed as
\[
   u^h(\bx,t) = \sum_{j=1}^{N_{\phi}} u^h_j(t) \phi_j (\bx) .
\]
Using $\bU = {(u^h_1,\dotsc, u^h_{N_{\phi}})}^T$, \cref{eq:FEM:i,eq:FEM:ii} can be written into a matrix form
\begin{equation}
   M \bU_t = - A \bU, \quad \bU(0) = \bU_0,
   \label{eq:fem:system}
\end{equation}
where the mass and stiffness matrices $M$ and $A$ are defined by
\[
   M_{ij} = \int_\Omega \phi_i \phi_j \dx
   \quad \text{and} \quad
   A_{ij} = \int_\Omega \nabla \phi_i \cdot \Diff \nabla \phi_j \dx
\]
for all $i, j = 1, \dotsc, N_{\phi}$.
We further assume that surrogate mass matrices $\tM$ considered throughout the paper satisfy
\begin{enumerate}[label=(M\arabic*)]
\item The reference element matrix $\tM_{\hat{K}}$
   is symmetric positive definite.\label{item:M1}
\item The element matrix $\tM_K$ satisfies
   $\tM_K = |K| \tM_{\hat{K}}$.\label{item:M2}
\end{enumerate}
For example, \cref{item:M1,item:M2} are satisfied for any mass lumping by means of numerical quadrature with positive weights.

\begin{lemma}[{\cite[Theorem~3.1]{HuaKamLan13b}}]\label{thm:stability}
For a given explicit RK method with the polynomial stability function $R$ and a symmetric positive definite surrogate matrix $\tM$ that satisfies\footnote{In the following, the less-than-or-equal-to sign for matrices means that the difference between the right-hand side and left-hand side terms is positive semidefinite.} $c_1\tM\le M\le c_2\tM$ for some positive constants $c_1$ and $c_2$, the finite element approximation $u^{h}_n$ at $t_n = n \tau$ satisfies
\[
   \Norm{u^{h}_{n}}_{L^2(\Omega)} 
      \le \sqrt{\frac{c_2}{c_1}} \Norm{u^{h}_{0}}_{L^2(\Omega)}
   \quad \text{and} \quad
   \NormE{u^{h}_{n}} \le \NormE{u^{h}_{0}},
\]
if the time step $\tau$ is chosen such that
\[
   \max_i \Abs{R\left(-\tau\lambda_i\left(\tM^{-1}A\right)\right)} \le 1 .
\]
\end{lemma}

This lemma is proven in~\cite{HuaKamLan13b} for the linear finite element discretization.
However, from the proof one can see that it is valid for any system in the form of \cref{eq:fem:system} with symmetric positive definite matrices $M$ and $A$.
Particularly, it can be used for the system \cref{eq:fem:system} resulting from the $\mathbb{P}_m$ finite element discretization.
In the following, we establish a series of lemmas for bounds on the stiffness and mass matrices $A$ and $\tM$ and then develop bounds for $\lambda_{\max} (\tM^{-1}A)$.

\begin{lemma}\label{thm:smatrix}
Let $\eta$ be the maximal number of basis functions per element.
Then the stiffness matrix $A$ and its diagonal part $A_D$ for $\mathbb{P}_m$ finite elements satisfy
\[
   A \le \eta A_D .
\]
\end{lemma}
\begin{proof}
Notice that for any positive semi-definite matrix $S$ and any vectors $\bu$ and $\bv$ we have
\[
   \bu^T S \bv + \bv^T S \bu \le \bu^T S \bu + \bv^T S \bv.
\]
From this,
\begin{align*}
   \bu^T A \bu
   &= \sum_{i,j} \int_\Omega {\left(u_i \nabla \phi_i\right)}^T \Diff
      \left(u_j \nabla\phi_j\right) \dx
   \\
   &\le \sum_i \eta \int_\Omega {\left(u_i \nabla \phi_i\right)}^T \Diff
      \left(u_i \nabla\phi_i\right) \dx
   \\
   &= \eta \sum_i u_i^2 \int_\Omega \nabla \phi_i^T \Diff \nabla\phi_i \dx
   \\
   &= \bu^T \eta A_D \bu.
   \qedhere{}
\end{align*}
\end{proof}

\begin{lemma}\label{thm:smatrix:fdf}
Let $\hat\phi_i$ be the basis functions on the reference element that correspond to $\phi_i$ and
\[
   C_{H^1}  = \max_i \abs{\hat\phi_i}_{H^1(\hat{K})}^2.
\]
Then the diagonal entries $A_{ii}$ of the stiffness matrix $A$
are bounded by
\[
   A_{ii} 
   \le  C_{H^1} \sum\limits_{K \in \omega_i} \Abs{K}
      \max_{\bx \in K} \Norm{\FDF}_2.
\]
\end{lemma}
\begin{proof}
From the definition of the stiffness matrix we have
\[
   A_{ii}  
   = \int_\Omega \nabla\phi_i^T \Diff \nabla\phi_i \dx
   = \sum_{K \in \omega_i} \int_K \nabla \phi_i^T \Diff \nabla \phi_i \dx
   .
\]
Let $\hat{\nabla} = {\partial}/{\partial \bxi }$ be the gradient operator in $\hat{K}$.
The chain rule yields $\nabla = {(F_K')}^{-T} \hat{\nabla}$ and together with $\det(F_K') = \Abs{K}$ we obtain
\begin{align*}
   A_{ii} 
   &= \sum_{K \in \omega_i} \Abs{K} \int_{\hat{K}}
      \hat{\nabla} \hat{\phi}_i^T \FDF \hat{\nabla} \hat{\phi}_i \dxi\\
   &\le \sum_{K \in \omega_i} \Abs{K} \; 
      \norm{\hat{\nabla} \hat{\phi}_i}_{L^2(\hat{K})}^2 
      \max_{\bx \in K} \Norm{\FDF}_2\\
   &\le C_{H^1} \sum_{K \in \omega_i}
      \Abs{K} \max_{\bx \in K} \Norm{\FDF}_2
   .
   \qedhere{}
\end{align*}
\end{proof}

\begin{lemma}\label{thm:M:W}
Let $\tM$ be a surrogate $\mathbb{P}_m$ finite element mass matrix, $\hat\Lambda_{\tM}$ and $\hat\lambda_{\tM}$ be the largest and smallest eigenvalues of the surrogate mass matrix $\tM_{\hat{K}}$ on the reference element  and 
\[
   \mathcal{W} = \diag\left(\Abs{\omega_1},\dotsc,\Abs{\omega_{N_\phi}}\right).
\]
Then 
\begin{equation}
   \hat\lambda_{\tM} \mathcal{W} 
   \le \tM \le
   \hat\Lambda_{\tM} \mathcal{W} 
   .
   \label{eq:M:W}
\end{equation}
\end{lemma}
\begin{proof}
We have
\begin{align*}
   \bu^T \tM \bu 
      &= \sum_K \bu_K^T \tM_K \bu_K 
       = \sum_K \Abs{K} \bu_K^T \tM_{\hat{K}} \bu_K
      \\
      &\le  \sum_K \Abs{K} \hat\Lambda_{\tM} \Norm{\bu_K}_{2}^2
      = \hat\Lambda_{\tM} \sum_i u_i^2 \sum_{K\in\omega_i} \Abs{K}
      \\
      &= \hat\Lambda_{\tM} \sum_i u_i^2 \Abs{\omega_i}
       = \hat\Lambda_{\tM} \bu^T \mathcal{W} \bu.
\end{align*}
The lower bound can be obtained similarly.
\end{proof}

\begin{lemma}\label{thm:M12}
Let $\tM_1$ and $\tM_2$ be two surrogate mass matrices for $\mathbb{P}_m$ finite elements.
Then
\[
   \frac{\hat\lambda_{\tM_1}}{\hat\Lambda_{\tM_2}}
      \tM_2
      \le \tM_1 \le 
   \frac{\hat\Lambda_{\tM_1}}{\hat\lambda_{\tM_2}}
      \tM_2.
\]
\end{lemma}
\begin{proof}
Use \cref{thm:M:W} by applying \cref{eq:M:W} to $\tM_1$ and $\tM_2$.
\end{proof}

\begin{corollary}\label{thm:stability:M}
Let $\kappa(M_{\hat{K}})$ and $\kappa(\tM_{\hat{K}})$ be the condition numbers of the full and the surrogate reference element mass matrices.
Under the assumptions of \cref{thm:stability} we have
\[
  \Norm{u^{h}_{n}}_{L^2(\Omega)}
  \le \sqrt{\kappa(M_{\hat{K}})\kappa(\tM_{\hat{K}})}
   \Norm{u^{h}_{0}}_{L^2(\Omega)}
\]
and
\[
  \NormE{u^{h}_{n}} \le \NormE{u^{h}_{0}}.
\]
\end{corollary}
\begin{proof}
Use $\tM_1 = M$ and $\tM_2 = \tM$ in \cref{thm:M12} and apply \cref{thm:stability}.
\end{proof}

\begin{corollary}\label{thm:M:D}
The surrogate mass matrix $\tM$ for $\mathbb{P}_m$ finite elements and its diagonal part $\tM_D$  satisfy
\[
   \frac{1}{\kappa(\tM_{\hat{K}})}
      \tM_D
      \le \tM \le 
      \kappa(\tM_{\hat{K}})
      \tM_D.
\]
\end{corollary}
\begin{proof}
Using \cref{eq:M:W} with the canonical basis vector $\be_i$ implies 
\[
   \hat\lambda_{\tM} \mathcal{W}_{ii} \le \tM_{ii}
   \le \hat\Lambda_{\tM} \mathcal{W}_{ii},
\]
which gives
\[
   u_i \hat\lambda_{\tM} \mathcal{W}_{ii} u_i 
   \le u_i \tM_{ii} u_i \le u_i \hat\Lambda_{\tM} \mathcal{W} u_i
   \quad \text{for any $u_i$}.
\]
Since $\tM_D$ and $\mathcal{W}$ are diagonal matrices, this leads to
\begin{equation}
   \hat\lambda_{\tM} \mathcal{W} \le \tM_D \le \hat\Lambda_{\tM} \mathcal{W}
   .
   \label{eq:M:D}
\end{equation}
The statement now follows from \cref{thm:M12} with $\tM_1 = \tM$ and $\tM_2 = \tM_D$.
\end{proof}

Having obtained the preliminary bounds on the stiffness and mass matrices $A$ and $\tM$, we can now give the estimate for the largest eigenvalue of the system matrix $\tMA$ for $\mathbb{P}_m$ finite elements.

\begin{theorem}\label{thm:main:result}
The eigenvalues of $\tM^{-1}A$
are real and positive and the largest eigenvalue is bounded by
\begin{equation}
  \max_i \frac{A_{ii}}{\tM_{ii}}
  \le \lambda_{\max} \bigl( \tMA \bigr)
   \le \eta \, \kappa(\tM_{\hat{K}}) \max_i \frac{A_{ii}}{\tM_{ii}},
\label{eq:lmax}
\end{equation}
where $\eta$ is the maximal number of basis functions per element.
Further,
\begin{align}
   &\lambda_{\max} \bigl( \tM^{-1}A \bigr)
   \le
   \eta \frac{C_{H^1}}{\hat\lambda_{\tM}}
   \notag{}\\
   &\hspace{2.3em} \times \max_i \left\{ \sum\limits_{K \in \omega_i}
      \frac{\Abs{K}}{\Abs{\omega_i}}
      \max_{\bx \in K} \Norm{\FDF}_2 \right\}.
   \label{eq:main:result:geo}
\end{align}
\end{theorem}
\begin{proof}
Since $\tM$ and $A$ are symmetric positive definite, the eigenvalues of $\tM^{-1}A$ are real and positive.
The lower bound in \cref{eq:lmax} is obtained by using the canonical basis vectors $\be_i$ and the upper bound follows from \cref{thm:smatrix,thm:M:D},
\begin{align*}
   \lambda_{\max}(\tMA)
   &= \max_{\bv \neq 0} \frac{ \bv^T A \bv}{\bv^T \tM \bv}
   \le \max_{\bv \neq 0}\frac{ \bv^T \eta A_D \bv}
      {\bv^T \frac{1}{\kappa(\tM_{\hat{K}})} \tM_D \bv}
   \\
   &= \eta \, \kappa(\tM_{\hat{K}}) \max_i \frac{A_{ii}}{\tM_{ii}}.
\end{align*}
The geometric bound \cref{eq:main:result:geo} is a direct consequence of \cref{thm:smatrix,thm:smatrix:fdf,thm:M:W},
\begin{align*}
   &\hspace{0.3em}
   \lambda_{\max}(\tMA) = \max_{\bv \neq 0} \frac{ \bv^T A \bv}{\bv^T \tM \bv}
   \le \max_{\bv \neq 0} \frac{ \bv^T \eta A_D \bv}
      {\bv^T \hat\lambda_{\tM} \mathcal{W} \bv} \\
   &\hspace{1.75em}
   \le \eta \frac{C_{H^1}}{\hat\lambda_M}
      \max_i \left\{ \sum\limits_{K \in \omega_i}
      \frac{\Abs{K}}{\Abs{\omega_i}}
      \max_{\bx \in K} \Norm{\FDF}_2 \right\}.
\end{align*}
\end{proof}

\Cref{thm:main:result} can be used in combination with \cref{thm:stability} or \cref{thm:stability:M} to derive the stability condition of a given explicit Runge-Kutta scheme, as shown in the next example.

\begin{example}[Explicit Euler method]\label{sec:numerics}\label{ex:explicit:SRK}
The stability region of the explicit Euler method includes the real interval $[-2, 0]$.
\Cref{thm:stability} implies that the method is stable if
\[
   -2 \le -\tau\lambda_i(\tM^{-1}A) \le 0,
   \qquad i = 1, \dotsc, N_\phi.
\]
Using \cref{thm:main:result}, we conclude that the method is stable if the time step $\tau$ satisfies
\[
   \tau
   \le \frac{2}{\eta \, \kappa(\tM_{\hat{K}})} \min_i \frac{\tM_{ii}}{A_{ii}}
\]
or, in terms of mesh geometry,
\[
   \tau
   \le
   \frac{2 \hat\lambda_{\tM_{\hat{K}}}}{\eta \,  C_{H^1}}
   \min_i {\left( \sum\limits_{K \in \omega_i}
      \frac{\Abs{K}}{\Abs{\omega_i}}
      \max_{\bx \in K} \Norm{\FDF}_2 \right)}^{-1}
   .
\]
\end{example}

\begin{remark}
   \Cref{thm:smatrix,thm:smatrix:fdf,thm:M:D} are very general and valid for any mesh, any $\Diff$ and any surrogate mass matrix $\tM$ satisfying \cref{item:M1,item:M2}.
More accurate bounds can be obtained if more information is available about the mesh or the stiffness and mass matrices.

For example, if $A$ is an M-matrix, then the Gershgorin circle theorem yields $\lambda_{\max}(A) \le 2 \max_i A_{ii}$~\cite[Remark~2.2]{HuaKamLan13b} and therefore $\eta$ in \cref{thm:main:result} can be replaced by $2$.

If $\tM = M$ (no mass lumping), then, instead of estimating $M_D$ through~\cref{eq:M:D}, a direct calculation for the standard $\mathbb{P}_m$ finite elements yields
\[
   M_D = C_{L^2} \mathcal{W},
   \qquad 
   C_{L^2} = \diag \left(\norm{\hat\phi_1}_{L^2}^2, 
      \dotsc, \norm{\hat\phi_{N_\phi}}_{L^2}^2 \right),
\]
and
\[
   \hat\lambda_M C_{L^2}^{-1} M_D 
      \le M \le \hat\Lambda_M C_{L^2}^{-1} M_D,
\]
resulting in a slighly more accurate bound in \cref{thm:M:D}.
For simplicity, in \cref{thm:smatrix:fdf} we used $C_{H^1}  = \max_i \abs{\hat\phi_i}_{H^1(\hat{K})}^2$.
A slightly more accurate bound can be derived if we use
\[
   C_{H^1}  
   = \diag\left(\abs{\hat\phi_1}_{H^1(\hat{K})}^2,
      \dotsc, \abs{\hat\phi_{N_\phi}}_{H^1(\hat{K})}^2\right).
\]
\end{remark}

\section{Summary and conclusion}\label{sec:conclusion}

\Cref{thm:main:result} states that the largest eigenvalue of the system matrix and, thus, the largest permissible time step can be bounded by a term depending only on the number of mesh elements and the alignment of the mesh with the diffusion matrix.

The bound in terms of matrix entries is tight within a small factor which depends only on the dimension and the choice of the reference element and basis functions but is independent of the mesh or the coefficients of the IBVP.\@
This is valid for any Lagrangian $\mathbb{P}_m$ finite elements with $m \ge 1$.

A similar result is obtained by Zhu and Du~\cite[Theorem~3.1]{ZhuDu14}.
In our notation, it can be written as
\begin{multline}
   \lambda_{\max}(\MA) \\ \lesssim
      \max\limits_K \Big\{ 
         \max_{\bx\in K} \lambda_{\max}(\Diff) \Norm{\FF}_2 
         \Big\}
         .
   \label{eq:zhudu}
\end{multline}
The significant difference between the new bound \cref{eq:main:result:geo} and the bound \cref{eq:zhudu} is the factor which represents the interplay between the mesh geometry and the diffusion matrix,
\[
   \max_{\bx \in K} \Norm{\FDF}_2
\]
vs.\
\[
   \max_{\bx \in K} \lambda_{\max}(\Diff) \Norm{\FF}_2
   .
\]
For isotropic $\Diff$ or isotropic meshes both terms are comparable.
However, the former is smaller than the latter in general.
In particular, if both $\Diff$ and $K$ are anisotropic, then the difference between \cref{eq:main:result:geo,eq:zhudu} can be very significant (see~\cite[Sect.~4.4]{HuaKamLan13b} for a numerical example in case of $\mathbb{P}_1$ finite elements).
In this sense, \cref{thm:main:result} can be seen either as a generalization of~\cite{HuaKamLan13b} to $\mathbb{P}_m$~($m \ge 2$) finite elements or as a more accurate version of~\cite{ZhuDu14} for anisotropic meshes and general diffusion coefficients.

Finally, we would like to point out that a similar result can be established for $p$-adaptive finite elements without major modifications.



\begin{thebibliography}{99}
\parskip1.0ex

\bibitem{DuWanZhu09}
{Q.~Du, D.~Wang, and L.~Zhu}, \emph{On mesh geometry and stiffness matrix
  conditioning for general finite element spaces},
  SIAM J. Numer. Anal., 47 (2009), pp.~1421--1444,
  \url{https://doi.org/10.1137/080718486}.

\bibitem{Fri73}
{I.~Fried}, \emph{Bounds on the spectral and maximum norms of the finite
  element stiffness, flexibility and mass matrices},
  Internat. J. Solids and Structures, 9 (1973), pp.~1013--1034,
  \url{https://doi.org/10.1016/0020-7683(73)90013-9}.

\bibitem{GraMcL06}
{I.~G. Graham and W.~McLean}, \emph{Anisotropic mesh refinement: the
  conditioning of {G}alerkin boundary element matrices and simple
  preconditioners}, SIAM J. Numer. Anal., 44 (2006), pp.~1487--1513,
  \url{https://doi.org/10.1137/040621247}.

\bibitem{HuaKamLan13b}
W.~Huang, L.~Kamenski, and J.~Lang,
\emph{Stability of explicit time integration schemes for finite element
   approximation of linear parabolic equations on anisotropic meshes},
  \href{http://wias-berlin.de/publications/wias-publ/run.jsp?template=abstract&type=Preprint&year=2013&number=1869}{WIAS Preprint No.~1869} (2013).
  Appeared in SIAM J. Numer. Anal., 54 (2016), pp.~1612--1634, as
  \emph{Stability of explicit one-step methods for {P}1-finite element
  approximation of linear diffusion equations on anisotropic meshes},
  \url{https://doi.org/10.1137/130949531},
  \url{https://arxiv.org/abs/1602.08055}.

\bibitem{KamHuaXu13}
{L.~Kamenski, W.~Huang, and H.~Xu}, \emph{Conditioning of finite element
  equations with arbitrary anisotropic meshes},
  Math. Comp., 83 (2014), pp.~2187--2211,
  \url{https://doi.org/10.1090/S0025-5718-2014-02822-6},
  \url{https://arxiv.org/abs/1201.3651}.

\bibitem{ZhuDu11}
{L.~Zhu and Q.~Du}, \emph{Mesh-dependent stability for finite element
  approximations of parabolic equations with mass lumping},
  J. Comput. Appl.  Math., 236 (2011), pp.~801--811,
  \url{https://doi.org/10.1016/j.cam.2011.05.030}.

\bibitem{ZhuDu14}
{L.~Zhu and Q.~Du}, \emph{Mesh dependent stability and condition number
  estimates for finite element approximations of parabolic problems},
  Math.  Comp., 83 (2014), pp.~37--64,
  \url{https://doi.org/10.1090/S0025-5718-2013-02703-2}.

\end{thebibliography}
\newpage{}

\ifx\undefined\bysame{}
\newcommand{\bysame}{\leavevmode\hbox to3em{\hrulefill}\,}
\fi{}

\end{document}